\documentclass[11pt]{amsart}

\usepackage{amsmath,enumerate,enumitem}
\usepackage{amssymb}
\usepackage{amsthm}
\usepackage{graphicx}
\usepackage{tikz}
\usepackage{xcolor}
\usepackage{hyperref}
\usepackage{mathtools}
\usepackage{url}
\mathtoolsset{showonlyrefs}
\usepackage{enumitem}

\allowdisplaybreaks

\usepackage[margin=3.5cm]{geometry}

\def\E{\mathbb{E}}
\def\var{\mathbb{Var}}

\def\bin{{\rm Bin}}

\def\kotecky{Koteck\'y}
\def\E{\mathbb{E}}

\def\Z{\mathbb{Z}}

\def\eps{\varepsilon}

\def\cD{\mathcal{D}}
\def\cC{\mathcal{C}}
\def\cI{\mathcal{I}}
\def\cE{\mathcal{E}}
\def\cO{\mathcal{O}}
\def\cP{\mathcal{P}}
\def\cL{\mathcal{L}}
\def\cT{\mathcal{T}}

\def\cG{\mathcal {G}}

\def\1{\mathbf{1}}

\def\lam {\lambda}

\def\var{\text{var}}

\newtheorem*{theorem*}{Theorem}
\newtheorem{theorem}{Theorem}
\newtheorem{lemma}[theorem]{Lemma}
\newtheorem{cor}[theorem]{Corollary}
\newtheorem{defn}[theorem]{Definition}
\newtheorem*{defn*}{Definition}
\newtheorem{prop}[theorem]{Proposition}
\newtheorem*{prop*}{Proposition}

\newtheorem*{conj*}{Conjecture}

\newtheorem*{fact*}{Fact}
\newtheorem{fact}[theorem]{Fact}

\begin{document}
\title{Independent sets in the hypercube revisited}

\author{Matthew Jenssen}
\address{School of Mathematics, University of Birmingham, Birmingham, UK}
\email{m.jenssen@bham.ac.uk.}
\author{Will Perkins}
\address{Department of Mathematics, Statistics, and Computer Science \\University of Illinois at Chicago\\851 S. Morgan, Chicago, IL}
\email{math@willperkins.org.}

\subjclass{05C30, 05C31, 82B20}
\keywords{independent sets, hypercube, cluster expansion}

\date{\today}

\maketitle

\begin{abstract}
We revisit Sapozhenko's classic proof on the asymptotics of the number of independent sets in the discrete hypercube $\{0,1\}^d$ and Galvin's follow-up work on weighted independent sets.  We combine Sapozhenko's graph container methods with the cluster expansion and abstract polymer models, two tools from statistical physics, to obtain considerably sharper asymptotics and  detailed probabilistic information about the typical structure of (weighted) independent sets in the hypercube.  These results  refine those of Korshunov and Sapozhenko and  Galvin, and answer several questions of Galvin.  
\end{abstract}

\section{Introduction}
\label{secIntro}

Let $Q_d$ denote the discrete hypercube of dimension $d$: the graph with vertex set $\{0,1\}^d$ with edges between vectors that differ in exactly one coordinate.    An independent set in a graph $G$ is a set of vertices that induce no edges.  Let $i(G)$ denote the number of independent sets of $G$.  

Korshunov and Sapozhenko proved the following result on the number of independent sets of the hypercube. 

\begin{theorem}[Korshunov and Sapozhenko~\cite{korshunov1983number}]
\label{thmSapo}
\[ i(Q_d) =(1+o(1)) \cdot 2 \sqrt{e}\cdot   2^{2^{d-1}} \]
as $d \to \infty$.  
\end{theorem}
A beautiful and influential proof of Theorem~\ref{thmSapo} was later given by Sapozhenko in~\cite{sapozhenko1987number}.  See~\cite{galvin2006independent} for an exposition of this proof. 

One of our main results in this paper will be to reinterpret Sapozhenko's proof in terms of the \textit{cluster expansion} from statistical physics.  This allows us to compute additional terms in the asymptotic expansion of $i(Q_d)$ among other things.  For instance, we can compute the asymptotics to the third order in $2^{-d}$. 
\begin{theorem}
\label{ThmRefined}
{\small
\[
i(Q_d) 
=  2\sqrt{e} \cdot 2^{2^{d-1}}  \left(   1+ \frac{3d^2 - 3d -2  }{8 \cdot 2^{d}}  +  {  \frac{243 d^4 - 646 d^3 - 33 d^2 + 436d +76}{384 \cdot 2^{2d}}  }+ O\left(d^ 6 \cdot 2^{-3d} \right)   \right )
\]
}
as $d \to \infty$.  
\end{theorem}
More generally, we give a formula and an algorithm for computing the asymptotics to arbitrary order in $2^{-d}$. 

Theorem~\ref{thmSapo} (along with Sapozhenko's techniques)  provided the first glimpse of a rich landscape of phenomena concerning independent sets in $Q_d$.  To describe the phenomena we take the perspective of statistical physics.  The independence polynomial of the hypercube is  
\begin{align*}
Z(\lam) &= \sum_{I \in \mathcal I(Q_d)} \lam^{|I|} \,,
\end{align*}
where $\cI(Q_d)$ is the set of all independent sets of $Q_d$.
In particular, $Z(1) = i(Q_d)$. 
The independence polynomial is the partition function of the \textit{hard-core model} from statistical physics: a probability distribution on independent sets weighted by the fugacity parameter $\lam$.  This distribution is defined by 
\begin{align*}
\mu (I) &= \frac{\lam^{|I|}}{Z(\lam)} \,.
\end{align*}
The hard-core model (or hard-core lattice gas) is a simple model of a gas, and in statistical physics it is most commonly studied on the integer lattice $\Z^d$. As is common in the literature, we will refer to vertices contained in an independent set drawn from the hard-core model as `occupied'.

Let $\cE \subset V(Q_d)$ be the set of `even' vertices of the hypercube whose coordinates sum to an even number and let $\cO \subset V(Q_d)$ be the `odd' vertices whose coordinates sum to an odd number.   
We note that $Q_d$ is a bipartite graph with bipartition $(\cE, \cO)$. 
 Kahn~\cite{kahn2001entropy} showed that for constant $\lam$, typical independent sets drawn from $\mu$ contain either mostly even vertices or mostly odd vertices, and thus the hard-core model on $Q_d$ exhibits a kind of `phase coexistence' in the language of statistical physics.

By generalizing Sapozhenko's techniques, Galvin~\cite{galvin2011threshold} was able to describe the typical structure of independent sets drawn from $\mu$ in  greater detail and for a wider range of parameters  $\lam$.  We need two  definitions to describe these results.
\begin{defn}
For an independent set $I \in \cI(Q_d)$, we say $\cE$ is the \textit{minority side} of the bipartition if $|\cE \cap I| < |\cO \cap I|$ and the \textit{majority side} otherwise.  If $\cE$ is the minority side, then $\cO$ is the majority side and vice versa.
\end{defn}
\begin{defn}
A set $S\subseteq \cE$ (or $\cO$), is $2$-linked if the subgraph of $Q_d$ induced by the vertex set $S\cup N(S)$ is connected; in other words, $S$ is connected in the graph $Q_d^2$ (the square of the graph $Q_d$).
 \end{defn}
Galvin showed that for the hard-core model on $Q_d$ at fugacity $\lam = 1+ s/d$ with $s$ constant, the number of occupied vertices on the minority side  is asymptotically distributed as a Poisson random variable with mean $e^{-s/2}/2$ and with high probability (whp) all $2$-linked components of occupied vertices on the minority side are of size $1$~\cite[Theorem 1.4]{galvin2011threshold}.  He conjectured that there is in fact a series of thresholds at which $2$-linked components of size $t$ emerge in a Poisson fashion and asked as an open problem for the distribution of occupied $2$-linked components of size $t$ on the minority side for all $t$.  

Here we prove his conjecture and answer his question in a strong form. We show that the emergence of a $2$-linked occupied component of size $t$ on the minority side has a sharp threshold at $\lam_t=2^{1/t}-1$ and we identify precisely the scaling window about this threshold. We also essentially determine the asymptotic joint distribution of the number of such components (see Theorem~\ref{thmDistribution}).
\begin{theorem}
\label{corEmergence}
For $t \ge 1$ fixed, let 
\begin{align*}
\lam_t(d) &=  2^{1/t} -1 + \frac{  2^{1+1/t}(t-1)\log d}{td} + \frac{s(d)}{d} \,.
\end{align*}
Then for the hard-core model on $Q_d$ at fugacity $\lam_t(d) $:
\begin{itemize}
\item if $s(d) \to \infty$ as $d \to \infty$ then whp there are no $2$-linked occupied components of size $t$ on the minority side;
\item if $s(d) \to - \infty$ then whp there are $2$-linked occupied components of size $t$ on the minority side;
\item if $s(d)$ tends to a constant $s$ then the distribution of the number of $2$-linked occupied components of size $t$ on the minority side converges to a Poisson distribution with mean
\[
e^{-st 2^{-1/t}} 2^{2-2/t-t}  (2^{1/t}-1)^t \sum {|\text{Aut (T)}|}^{-1}
\]
where the sum is over all trees $T$ on $t$ vertices and $\text{Aut} (T)$ denotes the automorphism group of the tree $T$.
\end{itemize}
\end{theorem}

In fact we prove much more detailed probabilistic results.  Define the \textit{defect type} of a $2$-linked component $S$ of $\cE$ or $\cO$ to be the isomorphism class of the induced subgraph $Q_d^2[S]$.  In particular there is a unique defect type of size $1$ (an isolated vertex), a unique defect type of size $2$ (two vertices at distance $2$ in $Q_d$), but two defect types of size $3$: $3$ vertices whose distance-$2$ graph forms a clique and $3$ vertices whose distance-$2$ graph forms a path.  For a given defect type $T$, let $X_{T}$ be the random variable that counts the number of $2$-linked occupied components of type $T$ on the minority side in the hard-core model on $Q_d$.  Let $m_{T} = \E X_{T}$ and $\sigma^2_{T} = \var(X_{T})$ (for the asymptotics of $m_T$ and $\sigma^2_T$ see Lemma~\ref{lemDefectTypes} and Corollary~\ref{lemMeanVar} below). 

We determine the limiting distribution of the number of each type of defect and show that the number of defects of different types are asymptotically independent. 
\begin{theorem}
\label{thmDistribution}
There is a constant $C_0>0$ such that if $\lam \ge C_0 \log d / d^{1/3}$ and $T$ is a defect type then the following holds. If $T$ and $\lam$ are such that $m_{T} \to \rho$ as $d \to \infty$ for some constant $\rho >0$, then
\begin{align*}
X_{T} \Rightarrow \text{Pois}(\rho) \,,
\end{align*}
where `$\Rightarrow$' denotes convergence in distribution.
If $T$ and $\lam$ are such that $m_{T} \to \infty$ as $d \to \infty$, then
\begin{align*}
\tilde X_{T} = \frac{X_{T}-m_{T}}{\sigma_{T}} \Rightarrow N(0,1) \,.
\end{align*}

\noindent
Moreover, suppose we have  two finite sets of defect types $\cT_1, \cT_2$ so that for each $T \in \cT_1$, there exists $\rho_T >0$ so that $m_T \to \rho_T$, and for each $T \in \cT_2$, $m_T \to \infty$.  Then the collection of random variables $\{ X_T \}_{T \in \cT_1} \cup \{ \tilde X_{T} \}_{T \in \cT_2}$ converges in distribution to a collection of independent Poisson and standard normal random variables. 
\end{theorem}

We remark that the condition that $\lam \ge C_0 \log d / d^{1/3}$ is a technical requirement of a container lemma due to Galvin which is a key ingredient in our proofs (see Lemma~\ref{lemGabEstimate} below). We expect that Theorem~\ref{thmDistribution} in fact extends to the range $\lam>(1+\Omega(1))\log d /d$.

There is a close connection between computing accurate estimates of the partition function and deriving probabilistic information about the hard-core model.  As a key step in proving his probabilistic results, Galvin gave a significant generalization of Theorem~\ref{thmSapo} to counting weighted independent sets in the hypercube; that is, computing the asymptotics of $Z(\lam)$ for general $\lam$.  
\begin{theorem}[Galvin~\cite{galvin2011threshold}]
\label{thmGalvin}
For $\lam \ge \sqrt{2} -1 + \frac{(\sqrt{2} +\Omega(1))   \log d}{   d}  $,
\begin{align}
\label{eqGalvin1}
Z(\lam) &= (2+o(1)) \cdot \exp \left[\frac{\lam}{2}\left(\frac{2}{1+\lam}\right)^d  \right] \cdot (1+\lam)^{2^{d-1}} \,.
\end{align}
Moreover, there is a constant $C_0 >0$ so that for $\lam \ge C_0 \log d/ d^{1/3} $,
\begin{align}
\label{eqGalvin2}
Z(\lam) &= 2 (1+\lam)^{2^{d-1}} \cdot \exp \left( \frac{\lam}{2}\left(\frac{2}{1+\lam}\right)^d (1+o(1))  \right) \,.
\end{align}
\end{theorem}
The formula~\eqref{eqGalvin1} generalises Theorem~\ref{thmSapo}, and determines the asymptotics of $Z(\lam)$ for $\lam > \sqrt{2} -1$, while the formula~\eqref{eqGalvin2} finds the asymptotics of $\log Z(\lam)$ for $\lam =\Omega(\log d/ d^{1/3})$.  

Our techniques based on the cluster expansion will allow us to sharpen Theorem~\ref{thmGalvin} considerably: we find a formula that can be used not only to determine the asymptotics of $Z(\lam)$ for all constant $\lam$ but also to give an expansion of $\log Z(\lam)$ to arbitrary order in $2^{-d}$.

To write the formula we need some notation that comes from \textit{polymer models} in the statistical mechanics of lattice systems~\cite{kotecky1986cluster}. Before we introduce these notions formally, we describe some of the intuition underlying the proof of Theorem~\ref{ThmRefined} and the results to come.
An immediate lower bound on $Z(\lam)$ of $2(1+\lam)^{2^{d-1}}-1$ comes by considering the contribution from independent sets which lie entirely in one side of the bipartition of $Q_d$. We call the collection of independent sets which lie entirely in $\cE$ (or $\cO$) the even (odd) \emph{ground state}. Taking $\lam=1$, for example,  there is a constant factor gap between this trivial lower bound $i(Q_d)\ge 2\cdot 2^{2^{d-1}}-1$ and the correct asymptotics of Theorem \ref{thmSapo}. Therefore a constant proportion of independent sets do not belong to a ground state. However,  almost all independent sets are very close to a ground state independent set. Thus it is natural to describe independent sets in terms of their deviations from a ground state: given a subset $X\subseteq\cE$, let $p(X)$ denote the probability that an independent set $I$ chosen according to $\mu$ satisfies $I\cap\cE=X$. When $X$ is small, we think of it as a deviation from the odd ground state and note that the relative `cost' of such a deviation is
\begin{equation}
\label{eqWeightDefn}
\frac{p(X)}{p(\emptyset)}=\frac{\lam^{|X|}}{(1+\lam)^{|N(X)|}}\, .
\end{equation}
We denote this cost, or \textit{weight}, of a deviation $X$ by $w(X)$.  Crucially, the weight $w(X)$ factorises over the $2$-linked components of $X$, and so we define an \textit{even polymer} to be any $2$-linked subset of $\cE$, and define its weight by~\eqref{eqWeightDefn}. We define odd polymers similarly.

The language of polymer models allows us to relate the partition function $Z(\lam)$ to the partition function of a (multivariate) hard-core model on an auxiliary graph whose vertices are polymers and each polymer $S$ has its associated weight $w(S)$ as its fugacity. A key feature of this transformation is that while at large $\lam$ an independent set drawn from $\mu_{Q_d, \lam}$ is typically very structured, the corresponding deviations on the minority side are typically unstructured and behave almost independently.  Using  the cluster expansion, we are then able to extract almost complete probabilistic information from our model.  In particular it allows us to precisely quantify the contribution to $Z(\lam)$ from small deviations, and allows us to compute $\log Z(\lam)$ to essentially arbitrary accuracy.

The cluster expansion is a powerful and classical tool in the rigorous study of statistical mechanics. In our context, it is the multivariate Taylor expansion of the logarithm of the partition function of our auxiliary hard-core model. Studying this infinite series naturally leads to the question of convergence. Verifying the convergence of the cluster expansion amounts to showing that the number of polymers of a given weight is not too large.  This is where the container method of Sapozhenko comes in. In fact, all of the ingredients needed to show that this polymer model has a convergent cluster expansion are already present in Sapozhenko's work and Galvin's extensions.  In some sense Sapozhenko rediscovered the concept of a polymer model and computed the smallest order terms of the cluster expansion by hand.  Certainly the intuition behind the specific polymer model is clear in his work.

We now venture to make some of the above mentioned notions more concrete. 
Recall that an even/odd polymer is a $2$-linked subset of $\cE/\cO$ respectively.  The size of a polymer $S$, $|S|$, is the number of vertices in $S$. Since $Q_d$ exhibits symmetry between $\cE$ and $\cO$ we will restrict our attention to even polymers.  We say two even polymers $S_1, S_2$ are \textit{compatible} if $d_G(S_1, S_2) >2$; that is if $S_1 \cup S_2$ is not $2$-linked. Otherwise $S_1$ and $S_2$ are incompatible (and note that each polymer is incompatible with itself).  For a tuple $\Gamma$ of even polymers, the \textit{incompatibility graph}, $H(\Gamma)$, is the graph with vertex set $\Gamma$ and an edge between any two incompatible polymers.  An even \textit{cluster} $\Gamma$ is an ordered tuple of even polymers so that $H(\Gamma)$ is connected.  The \emph{size} of a cluster $\Gamma$ is $\|\Gamma \| = \sum_{S \in \Gamma} |S|$.   Let $\cC$ be the set of all even clusters and $\cC_k$ the set of all even clusters of size $k$.

Recall that for a polymer $S$, we define its weight to be 
\(
w(S) = \lam^{|S|}(1+\lam)^{-|N(S)|} \,.
\)
For a cluster $\Gamma$ we define 
\begin{align*}
w(\Gamma)&= \phi (H(\Gamma)) \prod_{S \in \Gamma} w(S) \,,
\end{align*}
where $\phi(H)$ is the \textit{Ursell function} of a graph $H$, defined by
\begin{align}
\label{eqUrsell}
\phi(H) &= \frac{1}{|V(H)|!} \sum_{\substack{A \subseteq E(H)\\ \text{spanning, connected}}}  (-1)^{|A|} \, .
\end{align}
Finally  for $k\ge 1$ we define
\begin{align}
\label{eqLkDef}
L_k &= \sum_{\Gamma \in \cC_k}  w(\Gamma)    \, .
\end{align}
Note that by symmetry $L_k$ would be identical if we had considered odd polymers and odd clusters instead.

We can now state our formula for $Z(\lam)$. 
\begin{theorem}
\label{thmHypercubeasymptotics}
Suppose $\lam \ge C_0 \log d/d^{1/3}$ and $\lam$ is bounded as $d \to \infty$. Then for all fixed $k \ge 1$,
\begin{align*}
Z(\lam) &= 2(1+\lam)^{2^{d-1}} \cdot \exp \left ( \sum_{j=1}^k L_j    + \eps_k \right)  \,
\end{align*}
 where
for each fixed $k$, $|\eps_k| = O \left ( \frac{2^d \lam^{k+1} d^{2k}}{(1+ \lam)^{d(k+1)} }\right )  $ as $d \to \infty$.   Moreover, $L_k$ can be computed in time $e^{O(k \log k)}$.
\end{theorem}

In fact, it is not essential that $\lam$ remain bounded as $d \to \infty$: a similar formula holds for all values of $\lam$ with an addition of $\exp(-2^d/d^4)$ to $\eps_k$, but for simplicity here we focus on the more interesting cases  when $\lam$ is bounded or tends to $0$.  

As a quick check, note that at $\lam=1$,  $L_1= 1/2$ since there are $2^{d-1}$ polymers of size $1$ and each has weight $2^{-d}$. Moreover, $\eps_2=O(d^42^{-2d})=o(1)$ and so Theorem~\ref{thmHypercubeasymptotics} implies that $i(Q_d) = 2\cdot 2^{2^{d-1}} e^{1/2 + o(1)}$, recovering Theorem~\ref{thmSapo}.

More generally, Theorem~\ref{thmHypercubeasymptotics} extends Theorem~\ref{thmGalvin}.  For instance, we can give a closed-form formula for the asymptotics of $Z(\lam)$ for any constant $\lam$.
\begin{cor}
\label{corZasymptotics}
For any fixed $t \ge 1$ and for $\lam \ge  2^{1/t} -1 + \frac{  2^{1+1/t}(t-1)\log d}{td} + \frac{\omega(1)}{d}$,
\begin{align}
\label{eqZasymp}
Z(\lam) &= (2+o(1)) 2^{2^{d-1}} \exp \left ( \sum_{j=1}^{t-1} L_j \right ) \,.
\end{align}
For example, if  $\lam \ge 2^{1/3} -1  + \frac{2^{7/3}    \log d}{   3d} + \frac{\omega(1)}{d} $, then
\begin{align}
\label{eqZexample}\hspace{0.8em}
Z(\lam) &= (2+o(1)) \cdot \exp \left[ \frac{\lam}{2}\left(\frac{2}{1+\lam}\right)^d \left( 1+ \frac{  (2\lam^2 +\lam^3) d(d-1) -2\lam    }{  4 (1+\lam)^d   } \right)   \right ]  (1+\lam)^{2^{d-1}}    \, .
\end{align}
\end{cor}

We find it rather remarkable how well the two tools from statistical physics, polymer models and the cluster expansion,  work with the graph container method, and we expect many further applications of this combination of methods. See~\cite{samotij2015counting} for a survey of the graph container method. In forthcoming work, Keevash and the first author \cite{Torushom} apply this combination of methods to resolve conjectures of Galvin and Engbers \cite{engbers2012h} and Kahn and Park \cite{kahn2018number} on the number of $q$-colourings of $Q_d$. As a future research direction, we ask whether these statistical physics tools can be used in conjunction with the method of hypergraph containers~\cite{balogh2015independent,saxton2015hypergraph} to derive finer asymptotics and probabilistic information in some of the many extremal combinatorics problems in which hypergraph containers have been deployed.

The paper is organised as follows: We introduce abstract polymer models and the cluster expansion in Section~\ref{secCluster}, and then specialise to the hypercube and prove Theorem~\ref{thmHypercubeasymptotics} in Section~\ref{secApproxThm}.  We prove the probabilistic results of Theorems~\ref{corEmergence} and~\ref{thmDistribution} in Section~\ref{SecProb}.  We explicitly compute $L_1, L_2$, and $L_3$ and prove Theorem~\ref{ThmRefined} in Section~\ref{secClusterWeights}.

\subsection*{Related work}

As Galvin remarked in~\cite{galvin2011threshold}, only a few properties of the hypercube $Q_d$ are needed in deriving Theorem~\ref{thmGalvin}; the same is true for Theorems~\ref{thmDistribution} and~\ref{thmHypercubeasymptotics}.  The essential properties are that the graph be bipartite and that some isoperimetric estimates hold (of the form of Lemma~\ref{lemIso} below).   In fact, using an  approach to approximate counting based on the cluster expansion~\cite{helmuth2018contours,JKP2}, one could obtain efficient algorithms to approximate the partition function $Z_G(\lam)$ and to sample from the hard-core model for a class of graphs with these properties.  The polymer models used in~\cite{JKP2,liao2019counting,cannon2019bipartite} to sample from the hard-core model on random regular bipartite graphs are very similar to the ones used here.  For a similar class of bipartite graphs Galvin and Tetali~\cite{galvin2006slow} showed that the Glauber dynamics Markov chain for sampling from the hard-core model exhibits slow mixing; that proof is also based on extending the ideas of Sapozhenko.

\section{Polymer models and the cluster expansion}
\label{secCluster}
Here we introduce the main tools we will use, abstract polymer models~\cite{gruber1971general,kotecky1986cluster} and the cluster expansion, both tools from statistical physics that have been used extensively to study phase diagrams of lattice spin models.  We have already encountered the terms `polymer' and `cluster' in the previous section. 
Indeed, the polymers from the introduction are concrete examples of a more general notion which we introduce now.

Let $\cP$ be a finite set whose elements we call `polymers'. We equip $\cP$ with a complex-valued weight $w(S)$ for each polymer $S$ as well as a symmetric and reflexive incompatibility relation between polymers. We write $S \nsim S'$ if polymers $S$ and $S'$ are incompatible. Let $\Omega$ be the collection of pairwise compatible sets of polymers from $\cP$, including the empty set of polymers.  Then the polymer model partition function is
\begin{align*}
\Xi(\cP) &= \sum_{\Gamma \in \Omega} \prod_{S\in \Gamma} w(S) \,,
\end{align*}
where the contribution from the empty set is $1$.

A \textit{cluster} is an ordered tuple of polymers whose incompatibility graph $H(\Gamma)$ is connected. 
Let $\cC$ be the set of all clusters. The \textit{cluster expansion} is the formal power series in the weights $w(S)$
\begin{align*}
\log \Xi(\cP) &= \sum_{\Gamma \in \cC} w(\Gamma) \,,
\end{align*}
where
\begin{align*}
w(\Gamma) &=  \phi(H(\Gamma)) \prod_{S \in \Gamma} w(S) \,,
\end{align*}
and $\phi(H)$ is the Ursell function as defined in~\eqref{eqUrsell}. 
In fact the cluster expansion is simply the multivariate Taylor series for $\log \Xi(\cP)$ in the variables $w(S)$, as observed by Dobrushin~\cite{dobrushin1996estimates}.  See also Scott and Sokal~\cite{scott2005repulsive} for a derivation of the cluster expansion and much more. 

A sufficient condition for the convergence of the cluster expansion is given by a theorem of \kotecky{ }and Preiss.
\begin{theorem}[\cite{kotecky1986cluster}]
\label{thmKP}
Let $f : \cP \to [0,\infty)$ and $g: \cP \to [0,\infty)$ be two functions.  Suppose that for all polymers $S \in \cP$, 
\begin{align}
\label{eqKPcond}
\sum_{S' \nsim S}  |w(S')| e^{f(S') +g(S')}  &\le f(S) \,,
\end{align}
then the cluster expansion converges absolutely.  Moreover, if we let $g(\Gamma) = \sum_{S \in \Gamma} g(S)$ and write $\Gamma \nsim S$ if there exists $S' \in \Gamma$ so that $S \nsim S'$, then for all polymers $S$,
\begin{align}
\label{eqKPtail}
 \sum_{\substack{\Gamma \in \cC \\  \Gamma \nsim S}} \left |  w(\Gamma) \right| e^{g(\Gamma)} \le f(S) \,.
\end{align}
\end{theorem}

We remark that one could simply take $g\equiv 0$ in \eqref{eqKPcond} in order to establish convergence of the cluster expansion. However, allowing $g$ to take non-zero values (thus strengthening \eqref{eqKPcond}) allows us to give strong tail bounds on the cluster expansion via \eqref{eqKPtail}.
This will allow us to show that certain truncations of the cluster expansion serve as good approximations to the logarithm of the partition function.

\section{Polymers  in the hypercube}
\label{secApproxThm}

We now return to our specific setting with polymers derived from the hard-core model on $Q_d$. These polymers will essentially be the same as those defined in Section~\ref{secIntro}.  Here we will study the cluster expansion of this polymer model in depth.

\subsection{Preliminaries}

We begin with some notation and lemmas from~\cite{galvin2011threshold}. 

For a set $A \subseteq \cE$ (and analogously for $A \subseteq \cO$), let $|A|$ denote the number of vertices of $A$, $N(A)$ be the set of neighbours of $A$, and $[A]= \{ v \in \cE : N(v) \subseteq N(A) \}$ the bipartite closure of $A$. 
 Clearly $|[ A]| \ge |A|$.   Let $$\cG(a,b) = \{ A \subseteq \cE: A \text{ 2-linked}, | [A]| = a, |N(A)| =b \}.$$

The following lemma of Galvin is based on the graph container method of Sapozhenko~\cite{sapozhenko1987number}.   
This is a key technical ingredient in~\cite{sapozhenko1987number,galvin2011threshold} and in the results of this paper. 
\begin{lemma}[\cite{galvin2011threshold}]
\label{lemGabEstimate}
There exist constants $C_0, C_1 >0$, so that for all $\lam \ge C_0\log d /d^{1/3}$, all $a \le 2^{d-2}$, 
\begin{align*}
\sum_{\substack{ A \in \cG(a,b) }} \frac{\lam^{|A|}}{(1+\lam)^{b}} \le 2^d \exp \left( - \frac{C_1 (b-a) \log d}{ d^{2/3}}  \right ) \,.
\end{align*}
\end{lemma}
In what follows, we will always assume that $\lam \ge  C_0\log d /d^{1/3}$ to allow us to apply Lemma~\ref{lemGabEstimate}.

We will also use the following isoperimetric estimates, which come from~\cite{galvin2003homomorphisms,korshunov1983number} but can also be found in~\cite{galvin2011threshold}.
\begin{lemma}
\label{lemIso}
Suppose $S \subseteq \cE$ (or $S \subseteq \cO$).  Then
\begin{enumerate}
\item If $|S| \le d/10$, then $|N(S)|\ge d |S| -2|S|^2 $.
\item If $|S| \le d^4$, then $|N(S)|\ge d |S|/10$. 
\item If $|S| \le 2^{d-2}$, then $|N(S)| \ge  \left(1 + \frac{1}{2\sqrt{d}} \right) |S|$. 
\end{enumerate}
\end{lemma}

We also make use of the following, from, e.g.~\cite{galvin2004phase}. 
\begin{lemma}\label{lemlink}
The number of $2$-linked subsets $S \subseteq \cE$ of size $t$ which contain a given vertex $v$ is at most $(ed^2)^{t-1}$.
\end{lemma}

\subsection{The defect polymer model}
We begin by fixing a side of the bipartition which we call the \emph{defect side}. Let us suppose this side is $\cE$ (the case where $\cO$ is the defect side will be identical).

We define a polymer to be a $2$-linked subset $S$ of the defect side in $Q_d$ so that $ |[ S]| \le 2^{d-2}$.  Let $\cP$ be the set of all such polymers (we will make use  of a subscript, as in $\cP_{\cE}$ or $\cP_{\cO}$, if we want to indicate which is the defect side).  Two polymers $S, S'$ are compatible if $S \cup S'$ is not $2$-linked.  Let $\Omega$ be the set of all pairwise compatible sets of polymers from $\cP$.  The weight functions are defined as 
 $$w(S) = \frac{\lam^{|S|}}{(1+\lam)^{|N(S)|}}\,.$$
  Let $\Xi = \Xi(\cP)$ denote the resulting polymer model partition function (and note that by symmetry $\Xi$ is the same regardless of the defect side).

The partition function $\Xi$ is the normalizing constant of a probability distribution $\nu$ on $\Omega$ defined by
\begin{align*}
\nu( \Gamma) &= \frac{ \prod_{S\in \Gamma} w(S) }{ \Xi(\cP) } \,.
\end{align*}
 
Using $\nu$ we can define a probability measure $\hat \mu$ on $\cI(Q_d)$ as follows:
\begin{enumerate}
\item With probability $1/2$ choose $\cD = \cE$ or $\cD = \cO$ to be the defect side.
\item Choose a polymer configuration $\Gamma \in \Omega_{\cD} $ from $\nu$ and assign all vertices of $\cup_{S\in \Gamma} S$ to be occupied on the defect side $\cD$.
\item For each vertex $v$ on the non-defect side that is not blocked by an occupied vertex on the defect side, include $v$ in the independent set independently with probability $\frac{\lam}{1+\lam}$.  
\end{enumerate}

The resulting distribution $\hat \mu$ is not exactly the hard-core model $\mu$ on $Q_d$, but we will show that the two distributions are very close in total variation distance.  Moreover, we will show that a scaling of the partition function $\Xi$ is a very good approximation of the hard-core partition function $Z(\lam)$.  Note that the defect side need not be the minority side: in step $3$ we may choose no vertices to be occupied opposite the defect side.  Nevertheless, we will show below that with very high probability the defect side is in fact the minority side of an independent set sampled according to $\hat \mu$ (Lemma \ref{lemDefectMinority} below).

\begin{lemma}
\label{lemPolyApprox}
We have
\begin{align}
\label{eqZXiapprox}
\left|  \log Z(\lam) - \log  \left[ 2 (1+ \lam)^{2^{d-1}}  \Xi \right ]  \right | = O\left( \exp(-2^{d}/d^4) \right ) \,.
\end{align}
  Moreover, 
 \begin{align*}
\| \hat \mu - \mu \|_{TV}  =O\left(\exp ( -2^{d}/d^4)\right) \,.
\end{align*}
\end{lemma}

We will prove Lemma~\ref{lemPolyApprox} after showing that the polymer model satisfies the  \kotecky--Preiss condition.   
Lemma~\ref{lemPolyApprox} allows us to work with $\Xi$ and $\nu$ to prove Theorems~\ref{thmDistribution} and~\ref{thmHypercubeasymptotics}.  
In particular, to prove Theorem~\ref{thmHypercubeasymptotics} we will approximate $\Xi$ by truncating the cluster expansion for $\log \Xi$ and exponentiating.  To prove Theorem~\ref{thmDistribution} we will prove the probabilistic statements for polymer configurations sampled from $\nu$ and then use Lemmas~\ref{lemPolyApprox} and~\ref{lemDefectMinority} to transfer these results to results about the minority side of an independent set drawn from $\mu$.

We define the truncated cluster expansion of $\log \Xi$ as 
\begin{align}
\label{eqTruncation}
T_k  &= \sum_{\substack{\Gamma \in \cC:\\ \| \Gamma \| < k }} w(\Gamma) \, .
\end{align}

We now show that condition~\eqref{eqKPcond} holds for the defect polymer model with appropriate choices of functions $f(\cdot)$ and $g(\cdot)$, and thus $T_k$ gives a good approximation to $\log \Xi$.
\begin{lemma}
\label{lemKPcondition}
For integers $d, k \ge 1$,  let
\begin{align*}
\gamma(d,k) &= \begin{cases}
\log(1+\lam)  (dk - 3 k^2)- 7 k\log d  \text{ if }  k \le \frac{d}{10}  \\
 \frac{d \log(1+\lam) k}{20} \text{ if }  \frac{d}{10} < k \le d^4\\
\frac{k}{d^{3/2}} \text{ if }  k > d^4 \,.
\end{cases} 
\end{align*}
Then for $d$ sufficiently large
\begin{align}
\label{eqClusterBound}
\sum_{\substack{\Gamma \in \cC \\ \|\Gamma \| \ge k}} |w(\Gamma)| &\le d^{-3/2} 2^{d-1} e^{-\gamma(d,k)}  \,.
\end{align}
In particular, for $k$ fixed and $d$ sufficiently large,
\begin{align}
\label{eqTruncateBound}
| T_k - \log \Xi | &\le   d^{7k-3/2} 2^d (1+\lam)^{-dk +3k^2}   \,.
\end{align}
\end{lemma}

\begin{proof}
Let $g: \cP \to [0, \infty)$ be defined by $g(S) = \gamma(d,|S|)$ and define $f: \cP \to [0,\infty)$ by $f(S) = |S| / d^{3/2}$.  We will show that the \kotecky--Preiss condition~\eqref{eqKPcond} holds. That is, for every $S \in \cP$,
\begin{align}
\label{eqKPcheck}
\sum_{S' \nsim S} w(S') e^{d^{-3/2} |S'| + g(S')} &\le |S| / d^{3/2}  \, .
\end{align}
To prove this we will show that for all $v \in \cE$,
\begin{align}
\label{eqPolyKPestimate}
\sum_{S \ni v} w(S) e^{d^{-3/2} |S| +g(S)} \le \frac{1}{d^{7/2}} \,,
\end{align}
and this will suffice since $S' \nsim S$ if and only if $S' \ni v$ for some $v \in N^2(S)$ and $|N^2(S)| \le d^2 |S|$.  We will break up the sum according to the different cases of $\gamma(d,k)$.

First we sum over $S$ with $|S| \le \frac{d}{10}$. 
 We use the fact that for such $S$, $|N(S)| \ge  d |S| - 2 |S|^2$ by Lemma~\ref{lemIso}, and that there are at most $\exp(3 k \log d)$ $2$-linked sets $S$ of size $k$ that contain a fixed vertex $v$ by Lemma~\ref{lemlink}. 
\begin{align*}
\sum_{\substack{S \ni v\\|S|\le  \frac{d}{10}}} w(S) e^{f(S) +g(S)} &\le  \sum_{k=1}^{d/10} e^{3k \log d} \frac{\lam^k}{(1+\lam)^{dk - 2 k^2}} e^{k d^{-3/2} + \log(1+\lam) (dk -  3 k^2) - 7k \log d  } \\
&\le \sum_{k \ge 1} \exp \left(3k \log d + k \log \lam  +k d^{-3/2}  -k^2 \log(1+\lam) - 7k \log d    \right)   \\
&\le  \sum_{k \ge 1} \exp \left(-4k \log d  +k d^{-3/2}     \right)
\end{align*}
which is  at most $\frac{1}{3d^{7/2}}$ for $d$ large enough.

We next sum over $S$ with $\frac{d}{10} < |S| \le d^4$. We use the fact that for such $S$, $|N(S)| \ge d |S|/10$ by Lemma~\ref{lemIso}.  
\begin{align*}
\sum_{\substack{S \ni v\\ d/10<|S|\le d^4}} w(S) e^{f(S) +g(S)} &\le \sum_{k=d/10}^{d^4} e^{3k \log d} \frac{\lam^k}{(1+\lam)^{ d k/10}} e^{k d^{-3/2} +dk \log(1+\lam) /20} \\
&= \sum_{k=d/10}^{d^4} \exp\left (k \left(d^{-3/2}+ \log \lam +3 \log d -\frac{d \log(1+\lam)}{20 }   \right)   \right )  \, ,
\end{align*}
and so if $\lam \ge C_0 \log d /d$ and $d$ is large enough, then this sum is at most $\frac{1}{3d^{7/2}}$. 

Now turning to $S$ with $d^4< |S| \le 2^{d-2}$, we have that $|N(S)| \ge |S| (1+ 1/(2\sqrt{d}))$, and so
\begin{align*}
\sum_{\substack{S \ni v\\d^4 <|S| \le 2^{d-2}}}w(S) e^{f(S) +g(S)} &= \sum_{\substack{d^4 <a\le 2^{d-2} \\ (1+1/(2 \sqrt{d}))a \le b \le 2^{d-1}}} \sum_{\substack{S \ni v  \\ |[S]|=a, |N(S)| =b }}\frac{\lam^{|S|}}{(1+\lam)^b}  e^{ 2 |S| d^{-3/2}}    \\
&\le \sum_{\substack{d^4 <a\le 2^{d-2} \\ (1+1/(2 \sqrt{d}))a \le b \le 2^{d-1}}}  e^{ 2 a d^{-3/2}} \sum_{\substack{S \ni v  \\ |[S]|=a, |N(S)| =b }}\frac{\lam^{|S|}}{(1+\lam)^b}   \\
&\le  \sum_{\substack{a>d^4 \\ b \ge (1+1/(2 \sqrt{d}))a}} e^{ 2a d^{-3/2}}  d \exp \left( - \frac{C_1 (b-a) \log d}{ d^{2/3}}  \right ) \, ,
\end{align*}
where the last inequality comes from applying Lemma~\ref{lemGabEstimate}.  In the sum, we have $(b-a) \ge a/(2 \sqrt{d})$ and $a >d^4$, and so
\begin{align*}
\frac{2a}{d^{3/2}} + \log d - \frac{C_1 (b-a) \log d}{d^{2/3}} & \le -a  d^{-7/6}
\end{align*}
for large enough $d$, and so
\begin{align*}
\sum_{\substack{S \ni v\\d^4 <|S| \le 2^{d-2}}}w(S) e^{f(S) +g(S)} &\le \sum _{\substack{d^4 <a\le 2^{d-2} \\ (1+1/(2 \sqrt{d}))a \le b \le 2^{d-1}}} \exp ( -a  d^{-7/6} ) \\
&\le 2^d \sum _{a> d^4}  \exp ( -a  d^{-7/6} )  \\
&\le \frac{1}{3d^{7/2}}
\end{align*}
for $d$ large enough.  Putting the three bounds together gives~\eqref{eqKPcheck}.

To prove the lemma we now apply Theorem~\ref{thmKP}, applying~\eqref{eqKPtail} for the polymer $S$ containing the single vertex $v$ to obtain:
\begin{align*}
\sum_{\Gamma \in \cC, \Gamma \nsim v} |w(\Gamma)| e^{g(\Gamma)} &\le d^{-3/2}\,.
\end{align*}
Summing over all $v$ gives
\begin{align}\label{eqgtail}
\sum_{\Gamma \in \cC} |w(\Gamma)| e^{g(\Gamma)} & \le 2^{d-1} d^{-3/2} \,.
\end{align}
Since $\gamma(d,k)/k$ is non-increasing in $k$, we have, recalling that $g(\Gamma)=\sum_{S\in\Gamma}g(S)$ and $\|\Gamma \|=\sum_{S\in\Gamma}|S|$,

\[
g(\Gamma)=\sum_{S\in\Gamma} \gamma(d,|S|)\geq 
\sum_{S\in\Gamma} \frac{|S|}{\|\Gamma \|}\gamma(d, \|\Gamma \|)
=\gamma(d, \|\Gamma \|)\, .
\]
Keeping only terms in inequality \eqref{eqgtail} corresponding to clusters of size at least $k$, we have
\begin{align*}
\sum_{\substack{\Gamma \in \cC \\ \| \Gamma \| \ge k}} |w(\Gamma)| &\le d^{-3/2} 2^{d-1} e^{-\gamma(d,k)}  
\end{align*}
as desired. 
\end{proof}

The \kotecky--Preiss condition also allows us to  prove a simple large deviation result for the total size of all polymers in a random polymer configuration drawn from $\nu$. 

Suppose $X$ is a random variable whose moment generating function $\E e^{tX}$ is defined for $t$ in a neighbourhood of $0$. We will make extensive use of the \textit{cumulant generating function} of  $X$, defined as 
\begin{align*}
h_t(X) &= \log \E e^{t X}  \,,
\end{align*}
that is, the logarithm of the moment generating function.
\begin{lemma}
\label{lempolyLD}
Let $\mathbf \Gamma$ be a random configuration drawn from the distribution $\nu$.  Then with probability at least $1 - O\left(\exp ( - 2^{d}/d^{4}  )  \right)$, we have  
\begin{align*}
\| \mathbf  \Gamma \| &\le  2^d/d^2 \, ,
\end{align*}
where $\|  \mathbf  \Gamma \| = \sum_{S \in \mathbf \Gamma} |S|$. 
\end{lemma}
\begin{proof}
We introduce an auxiliary polymer model with modified polymer weights:
\begin{align*}
\tilde w(S) &= w(S) e^{ |S|  d^{-3/2}}  \,.
\end{align*}
Let $\tilde \Xi$ be the associated polymer model partition function.   Then $\log \tilde \Xi -  \log \Xi = h_t( \|\mathbf \Gamma \|)$ at $t=d^{-3/2}$ where  $\mathbf \Gamma$ is a random polymer configuration from the original polymer model. 

In the  proof of Lemma~\ref{lemKPcondition}, all of the estimates hold if we were to replace $f(S) = |S|/d^{3/2}$ by $\tilde f(S) = 2 |S|/d^{3/2}$.  Therefore the proof shows that the  \kotecky--Preiss condition holds for the polymer weights $\tilde w(S)$, and the functions $f(S),  g(S)$ as above.  Applying~\eqref{eqKPtail} and summing over all polymers of size $1$ gives
\begin{align*}
\log \tilde \Xi &\le \sum_{\Gamma \in \cC} | \tilde w(\Gamma)| \\
&\le 2^{d-1} d^{-3/2} e^{- \gamma(d,1)} \\
&\le 2^{d-1} d^{11/2} (1+\lam)^{3-d}  \,.
\end{align*}

Then since $\Xi \ge 1$, we have 
\begin{align*}
h_{d^{-3/2} }(\|\mathbf \Gamma\|)  &\le \log \tilde \Xi \\
&\le 2^{d-1} d^{11/2} (1+\lam)^{3-d}   \,.
\end{align*}
By Markov's inequality we have 
\begin{align*}
\Pr[\|\mathbf  \Gamma \| > 2^d/d^2] &\le e^{-t 2^d/d^2} \E e^{t \| \mathbf \Gamma \| } \,,
\end{align*}
and setting $t = d^{-3/2}$ gives
\begin{align*}
\Pr[\|\mathbf  \Gamma \| > 2^d/d^2] &\le \exp \left [ -\frac{2^d}{d^{7/2}}  +\frac{d^{11/2}  2^{d-1}}{(1+\lam)^{d-3}} \right ] \\
&\le \exp(-2^{d}/d^{4}) 
\end{align*}
for large enough $d$ since for $\lam \ge C_0 \log d/ d^{1/3}$, $(1+\lam)^d$ grows faster than any fixed polynomial in $d$.   
\end{proof}

This large deviation bound allows us to show that with very high probability over an independent set drawn from $\hat \mu$, the defect side is the minority side. 
\begin{lemma}
\label{lemDefectMinority}
With probability at least $1- O\left(\exp (-2^{d}/d^4)\right)$ over the random independent set $\mathbf I$ drawn from $\hat \mu$, the minority side of the bipartition is the defect side.  
\end{lemma}
\begin{proof}
Let $\cD$ and $\mathcal M$ denote the defect and minority side respectively selected under $\hat \mu$. By Lemma~\ref{lempolyLD} we have
\begin{align}
\Pr[\mathcal M\neq \mathcal D] &\le \Pr[\mathcal M\neq \mathcal D \mid \|\mathbf\Gamma \|\le 2^d/d^2]+ O\left(\exp ( - 2^{d}/d^{4} )  \right)\, .
\end{align}
Let us fix an element $\Gamma$ in the sample space of $\mathbf\Gamma$ with $\|\Gamma \|\leq 2^d/d^2$ and let $V(\Gamma):=\bigcup_{S\in\Gamma}S$.
Let $\mathbf X$ denote the size of the intersection of $\mathbf I$ with the non-defect side conditioned on the event that $\mathbf \Gamma=\Gamma$.
Note that $\mathbf X$ has a $\bin(M, \lam/(1+\lam))$ distribution where $M=2^{d-1}-|N(V(\Gamma))|\ge(1-2/d)2^{d-1}$.
We then have
\begin{align*}
\Pr[\mathcal M\neq \mathcal D \mid \mathbf\Gamma=\Gamma]
&\le
\Pr[\mathbf X\le \| \Gamma \|]\\
&\leq 
\Pr[\mathbf X\le 2^d/d^2]\\
&\leq 
\Pr[\mathbf X\le \E \mathbf X /2]\\
&\leq 
e^{-\E \mathbf X/8}\\
&\leq
\exp ( - 2^{d}/d^{4} )\, ,
\end{align*}
for $d$ sufficiently large. For the penultimate inequality we applied the Chernoff bound.
The result follows. 
\end{proof}

Now we can prove Lemma~\ref{lemPolyApprox}.

\begin{proof}[Proof of Lemma~\ref{lemPolyApprox}]
We say an independent set $I$ is captured by the odd polymer model if every $2$-linked component $S$ of  $\cO \cap I$ has $|[ S]| \le 2^{d-2}$ and captured by the even polymer model if every $2$-linked component $S$ of  $\cE \cap I$ has $|[ S] | \le 2^{d-2}$.  If we view $2 (1+\lam)^{2^{d-1}} \Xi$ as the sum of $(1+\lam)^{2^{d-1}} \Xi$ for $\Xi $ representing the odd polymer model and $(1+\lam)^{2^{d-1}} \Xi$ for $\Xi $ representing the even polymer model, then each $I$ that is captured by the odd polymer model contributes $\lam ^{|I|}$ to the first summand and each $I$ that is captured by the even polymer model contributes $\lam ^{|I|}$ to the second summand.  

Observe first that every $I \in \cI(Q_d)$ is captured by either the odd or the even polymer model.  Indeed suppose not, then there exists $I \in \cI(Q_d)$ which contains a set $S \subseteq \cO$ with $|[S]| > 2^{d-2}$ and a set $S' \subseteq \cE$ with $|[S']| > 2^{d-2}$. It follows that $|N(S)|= |N([ S])|>2^{d-2}$ (since, for example, $Q_d$ contains a perfect matching). However then $N(S)\cap [S']\neq \emptyset$ and so $S\cap N(S')=S\cap N([S'])\neq \emptyset$, contradicting the fact that $I$ is an independent set.

It remains to bound the contribution to $2 (1+\lam)^{2^{d-1}} \Xi$ from independent sets that are counted twice.  That is, bound $\sum_ {I\in B} \lam^{|I|}$ where $B$ denotes the collection of independent sets that are captured by both the odd and even polymer models. However, any such independent set can be selected by $\hat \mu$ conditioned on the event that $ \mathcal M\neq \mathcal D$ (using the notation of Lemma~\ref{lemDefectMinority}). Letting $\mathbf I$ denote the independent set selected by $\hat \mu$ we have by Lemma~\ref{lemDefectMinority} that
\begin{align}
\label{eqPrUB}
\Pr [\mathbf I \in B \wedge  \mathcal M\neq \mathcal D]=\frac{\sum_ {I\in B} \lam^{|I|}}{2 (1+ \lam)^{2^{d-1}}  \Xi}=  O\left(\exp ( -2^{d}/d^4)\right)\, .
\end{align}
All together this gives the inequalities
\begin{align*}
\left(1 -  O\left(\exp ( -2^{d}/d^4)\right)\right) 2 (1+ \lam)^{2^{d-1}}  \Xi \le Z(\lam) \le 2 (1+ \lam)^{2^{d-1}}  \Xi \,,
\end{align*}
and so
\begin{align*}
 \log [2 (1+ \lam)^{2^{d-1}}  \Xi  ]  - O (\exp ( -2^{d}/d^4)) \le    \log Z(\lam) \le \log [2 (1+ \lam)^{2^{d-1}}  \Xi  ]  \,,
\end{align*}
which gives~\eqref{eqZXiapprox}. 
Recall one formula for the total variation distance between discrete probability measures:
\begin{align*}
\| \mu - \hat \mu \|_{TV} &= \sum_{I : \hat \mu( I) > \mu(I) } \hat \mu(I) - \mu(I) \,.
\end{align*}
The total variation distance bound is then immediate from~\eqref{eqPrUB} as the only independent sets that have higher probability under $\hat \mu$ than $\mu$ are those that are counted twice. 
\end{proof}

Now we can prove Theorem~\ref{thmHypercubeasymptotics}.  

\begin{proof}[Proof of Theorem~\ref{thmHypercubeasymptotics}]
First we prove the estimate $|L_r| = O \left ( \frac{2^d \lam^r d^{2(r-1)}}{(1+ \lam)^{dr -r^2} } \right )  $ for $r$ fixed. Let $\Gamma$ be a cluster with $\|\Gamma \|=r$. Since $V(\Gamma):=\bigcup_{S\in\Gamma}S$ is a 2-linked set of size at most $r$, there are $O(2^d d^{2(r-1)})$ possibilities for $V(\Gamma)$ by Lemma~\ref{lemlink}. Given a set $X\subseteq V(Q_d)$ of size at most $r$, there are at most a constant number of clusters $\Gamma$ of size $r$ such that $V(\Gamma)=X$. It follows that the number of clusters of size $r$ is $O(2^d d^{2(r-1)})$. By Lemma~\ref{lemIso}, the weight of any cluster of size $r$ is $O(\lam^r/(1+\lam)^{dr-r^2})$ (note that the Ursell function of a cluster of size $r$ is simply a constant). The claimed estimate on $|L_r|$ follows.

Let $k\ge 1$ be fixed.  
By \eqref{eqTruncateBound} we have that 
\begin{align}
\eps_k':=| T_{k+2} - \log \Xi | &\le   d^{7(k+2)-3/2} 2^d (1+\lam)^{-d(k+2) +3(k+2)^2}   \,
\end{align}
for $d$ sufficiently large and so 
\begin{align}
\Xi=\exp\left\{\sum_{j=1}^{k+1} L_j+\eps'_k \right\}\, .
\end{align}
It follows from Lemma~\ref{lemPolyApprox} that
\begin{align*}
Z(\lam)&=2 (1+ \lam)^{2^{d-1}} \exp\left\{\sum_{j=1}^{k} L_j+ L_{k+1}+\eps'_k +  O(\exp ( -2^{d}/d^4))\right\} \\
&=2 (1+ \lam)^{2^{d-1}} \exp\left\{\sum_{j=1}^{k} L_j+ \eps_k \right\} \,,
\end{align*}
where $|\eps_k|=O \left ( \frac{2^d \lam^{k+1} d^{2k}}{(1+ \lam)^{d(k+1)} } \right )$
(it is here we use that $\lam$ is bounded as $d \to \infty$). 

Finally we show that $L_k$ can be computed in time $e^{O(k \log k)}$.
Let $X$ be the family of all $2$-linked subsets of $\cE$ of size at most $k$ which contain the vertex $\underline 0= (0,\ldots,0)$. Given $S\in X$, we call a coordinate $i$ \emph{active} for $S$ if $x_i=1$ for some $x\in S$. We note that every $S\in X$ has at most $2k$ active coordinates. For $A \subseteq [d]$, we let $X_A$ denote the set of elements in $X$ whose set of active coordinates is precisely $A$. 

For $m\in[k]$, we will construct the list $\cL_m$ of all the elements $S\in X$ with $|S|=m$ and whose set of active coordinates are a subset of $[2k]$. We do so iteratively. Suppose we have constructed the list $\cL_m$. For a vertex $v\in V(Q_d)$, and $\{i,j\}\subseteq [d]$, let $v_{ij}$ denote the vertex of $Q_d$ obtained by flipping the $i$th and $j$th coordinate of $v$.
For each pair $\{i,j\}\subseteq [2k]$, $S\in \cL_m$ and $v\in S$,  add $S\cup \{v_{ij}\}$ to the list $\cL_{m+1}$ if $v_{ij}\notin S$. This procedure generates the whole list $\cL_{m+1}$ and shows that $|\cL_{m+1}|\le m\binom{2k}{2}|\cL_m|$ and so $|\cL_k|\le k!\binom{2k}{2}^k=e^{O(k \log k)}$. For $m\in[k]$ and $a\in[2k]$, let $\cL^a_m$ denote the subset of $\cL_m$ consisting of those sets whose active coordinates are precisely $[a]$. Note that we can generate the list $\cL^a_m$ in time $e^{O(k \log k)}$ by checking the elements of $\cL_m$ one by one. 

For a cluster $\Gamma$, we define the active coordinates of $\Gamma$ to be the active coordinates of the set $V(\Gamma)=\bigcup_{S\in \Gamma}S$. For fixed $a\in[2k]$ and $m\in[k]$, we generate the list $\cG_{m, k, a}$ of all clusters of size $k$ containing $\underline 0$ with active coordinates $[a]$ and $|V(\Gamma)|=m$. To do this we run through each $S\in \cL^a_m$ and create the list of clusters $\Gamma$ of size $k$ with $V(\Gamma)=S$. We claim that this can be done in time $e^{O(k \log k)}$. Recall that a cluster of size $k$ is an ordered set of polymers $(\gamma_1,\ldots, \gamma_\ell)$ such that $\sum_{i=1}^\ell |\gamma_i|=k$. Let us fix $S\in \cL^a_m$. Since there are at most $2^k$ ordered integer partitions of $k$, it suffices to show that for a fixed such partition $(m_1,\ldots, m_\ell)$ (so that $\sum_i m_i=k$) we may find, in time $e^{O(k \log k)}$, all clusters $(\gamma_1,\ldots, \gamma_\ell)$ for which $|\gamma_i|=m_i$ for all $i$ and $\bigcup_i \gamma_i=S$. To do this we can simply check each element of $\binom{S}{m_1}\times\ldots \times \binom{S}{m_\ell}$ (a set of size at most $e^{O(k\log k)}$) to see if it constitutes a legitimate cluster. 

By symmetry of coordinates and vertex transitivity of $Q_d$ we have
\begin{align}
L_k=2^{d-1}\sum_{j=1}^k \frac{1}{j}\sum_{a=1}^{2k}\binom{d}{a}\sum_{\Gamma\in \cG_{j, k, a}}w(\Gamma)\, .
\end{align}
Finally we note that by using an algorithm of Bj\"orklund, Husfeldt, Kaski, and Koivisto \cite[Theorem 1]{bjorklund2008computing}, we may calculate the Ursell function of a cluster $\Gamma\in \cG_{j, k, a}$ in time $e^{O(k)}$. Moreover for a set $S\in \cL_j^a$ where $j\in[k]$, we can calculate $|N(S)|$ in time $O(k^2)$. We can therefore calculate $w(\Gamma)$ in time $e^{O(k)}$.
\end{proof}

\section{Probabilistic properties via the cluster expansion}
\label{SecProb}

Here we use the cluster expansion to prove Theorems~\ref{corEmergence} and~\ref{thmDistribution} and Corollary~\ref{corZasymptotics}.  Using Lemmas~\ref{lemPolyApprox} and~\ref{lemDefectMinority} we see that up to $O(\exp(-2^{d}/d^4))$ total variation error, we may replace the minority side of an independent set drawn from $\mu$ with the defect side of an independent set drawn from $\hat \mu$; or in other words, a polymer configuration drawn from $\nu$. Thus in this section we will let $X_T$ denote the (random) number of polymers of type $T$ in a random polymer configuration $\mathbf \Gamma$ drawn from $\nu$, and prove the conclusions of Theorems~\ref{thmDistribution} and~\ref{corEmergence} for these random variables.   We will also assume throughout this section that $C_0 \log d/ d^{1/3} \le \lam \le 2$.  Theorem~\ref{thmDistribution} is vacuous if $\lam >2$ since $m_T \to 0$ for all types $T$ in that case; the formula~\eqref{eqZasymp} in Corollary~\ref{corZasymptotics} holds for $\lam>2$ by Theorem~\ref{thmGalvin}.

We begin with some preliminaries on cumulants of random variables. 
Recall the cumulant generating function of a random variable $X$, $h_t(X) = \log \E e^{tX}$. 
The $k$th \textit{cumulant} of $X$ is defined by taking derivatives of $h_t(X)$ and evaluating at $0$:
\begin{align*}
\kappa_k(X) &= \frac{ \partial^k h_t(X)}{\partial t^k} \Bigg|_{t=0} \,.  
\end{align*}
In fact the cumulants of $X$ are related to the moments of $X$ by a non-linear change of basis (see e.g.~\cite{leonov1959method}).  In particular, $\kappa_1(X) = \E X$ and $\kappa_2(X) = \var(X)$.   Moreover, if a random variable $X$ has a distribution determined by its moments, and if for a sequence of random variables $X_n$ we have $\lim_{n \to \infty} \kappa_k(X_n) = \kappa_k(X)$ for all $k \ge 1$, then $X_n$ converges to $X$ in distribution (denoted $X_n \Rightarrow X$). We will use this fact in conjunction with the following fact.

\begin{fact}
\label{factPoisNormal}
 If $X$ has a Poisson distribution with mean $m$, then $\kappa_k(X) = m$ for all $k$. 
If $X$ has a standard normal distribution (mean $0$, variance $1$) then $\kappa_1(X) =0$, $\kappa_2(X) = 1$, and $\kappa_k(X)= 0$ for all $k \ge 3$. 
\end{fact}

We also need a few preliminaries about defect types.  First, for fixed $t$ the number of defect types of size $t$ is bounded independent of $d$.  Let $\tau (S)$ denote the type of a polymer $S$.  The weight of a polymer $S$ is determined by $\tau(S)$, since $|N(S)|$ is determined by the number of edges of $S$ in the graph $Q_d^2[S]$. Let $w_T$ denote $w(S)$ for $S$ of type $T$.  Using Lemma~\ref{lemIso}, we have the simple bounds
\begin{align}
\label{eqWtBound}
 \frac{\lam^t}{(1+\lam)^{dt}} &\le w_T \le \frac{\lam^t}{(1+\lam)^{dt-2t^2  }}
\end{align}
for a type $T$ of size $t$ and $d$ large enough. 
   Note that for any fixed $k\ge1$ and any type $T$, we have $d^k  w_T \to 0$ as $d \to \infty$; that is, each polymer weight decays super-polynomially fast in $d$.   We denote by $n_T = n_T(d)$ the number of polymers of type $T$.
\begin{lemma}
\label{lemDefectTypes}
Let $T$ be a defect type of a fixed size $t$.  Then
   \begin{align}
   \label{eqdefect1}
\frac{2^{d-1}}{t} \le n_T \le \frac{2^{d-1}}{t} (ed^2)^{(t-1)}\, .
\end{align}
 Moreover, if $T$ is a tree defect type then
\[ n_T = (c_T+o(1)) 2^d d^{2t-2}  \]
where $c_T=2^{-t}|\text{Aut}(T)|^{-1}$
and if $T$ is not a tree then
\[ n_T = O(2^d d^{2t-3}) \,. \]
\end{lemma}   
 \begin{proof}
By the vertex transitivity of $Q_d$,
every vertex of $\cE$ (or $\cO$) is contained in the same number of polymers of type $T$. 
Let us denote this number by $n_{T,v}$ and note that 
$n_T=2^{d-1}n_{T,v}/t$. 
The lower bound in~\eqref{eqdefect1} follows from the fact that if there exists a polymer with type $T$, then certainly $n_{T,v}\ge1$ . The upper bound follows from the fact that 
 every vertex of $Q_d$ is contained in at most  $(ed^2)^{(t-1)}$ 
 $2$-linked sets of size $t$ by Lemma~\ref{lemlink}.

Since $T$ is a connected graph we may fix an ordering
 $(x_1,\ldots, x_t)$ of the vertices of $T$ so that 
 $T_i:=T[\{x_1,\ldots, x_i\}]$ is connected for all $i\in[t]$.
 We let $d_i$ denote the degree of the vertex $x_i$ in the graph $T_i$.
 
 We will construct an injective graph homomorphism 
 $\varphi: T\to Q_d^2[\cE]$
 recursively as follows.
 Suppose that we have constructed an injective graph homomorphism
  $\varphi_i: T_i\to Q_d^2[\cE]$
  for some $i\leq t-1$
  and let $m_i$ denote the number of such homomorphisms. 
  We now extend $\varphi_i$ to an injective graph homomorphism
  $\varphi_{i+1}: T_{i+1}\to Q_d^2[\cE]$.
  We consider two cases. 
  
  If $d_{i+1}>1$, then $\varphi_{i+1}(x_{i+1})$ 
  must lie in the joint neighbourhood of $\varphi_i(x)$ and $\varphi_i(y)$ for some $x,y\in V(T_i)$.
  For any pair of vertices $u,v\in \cE$ their codegree in $Q_d^2[\cE]$
  is at most $2(d-2)$ and so there are at most $2(d-2)$ choices for 
  $\varphi_{i+1}(x_{i+1})$ whence
  \begin{align}\label{mbd1}
  m_{i+1}\le 2(d-2) m_i\, .
  \end{align}
  
  Suppose now that $d_{i+1}=1$ and let $R_i$ denote the set of possible choices for $\varphi_{i+1}(x_{i+1})$.
  We note that $u\in R_i$ if and only if $u$ is adjacent to $\varphi_{i}(x_{i})$ and non-adjacent to $\varphi_{i}(x_{j})$ for $j<i$ in $Q_d^2[\cE]$. Again using the fact that the maximum codegree in $Q_d^2[\cE]$ is $2(d-2)$ it follows that $\binom{d}{2}-2(d-2)\le|R_i|\le \binom{d}{2}$. We then have that
   \begin{align}\label{mbd2}
  \left(\binom{d}{2}-2(d-2)\right) m_i \le m_{i+1}\le \binom{d}{2} m_i\, .
  \end{align}
 If $T$ is not a tree then $d_{i+1}>1$ for some $i\le t-1$. 
 It follows by ~\eqref{mbd1} and the upper bound of \eqref{mbd2}
 that $m_t=O(2^d d^{2(t-1)-1 })=  O(2^d d^{2t-3 })$. 
 The bound $n_T=O(2^d d^{2t-3 })$ follows from the fact that 
 $n_T=m_t/|\text{Aut}(T)|$ where $\text{Aut}(T)$ denotes the automorphism group of the graph $T$ (recall that $t$ is a constant).
 
 If $T$ is a tree then $d_{i+1}=1$ for all $i\le t-1$ and so by \eqref{mbd2}
 $m_t=(1+o(1))2^{d-1}d^{2(t-1)}2^{-(t-1)}$. The result follows.
 \end{proof}

Now fix a defect type $T$ and let $X_T$ be the number of polymers of type $T$ in $\mathbf \Gamma$.  We introduce modified polymer weights $\tilde w$, given by
\begin{align*}
\tilde w(S) = w(S) e^{t \mathbf 1 _{\tau(S)=T}} \,.
\end{align*}
Let $\tilde \Xi$ be the corresponding polymer model partition function.  Then we have
\begin{align*}
\E e^{t X_T} &=  \sum_{\Gamma} \nu(\Gamma) e^{t  \sum_{S \in \Gamma} \mathbf 1 _{\tau(S)=T}} \\
&= \frac{1}{\Xi} \sum_{\Gamma} \prod_{S \in \Gamma} w(S) e^{t \mathbf 1 _{\tau(S)=T}} \\
&= \frac{\tilde \Xi}{\Xi} \, ,
\end{align*}
and so
\begin{align*}
\kappa_k (X_T) &= \frac{\partial^k   }{\partial t^k}  \log \frac{ \tilde \Xi}{\Xi}  \Bigg |_{t=0}  \\
&= \frac{\partial^k  \log \tilde \Xi  }{\partial t^k}  \Bigg |_{t=0}    \,.
\end{align*}

If the cluster expansion for $\log \tilde \Xi$ converges absolutely, we can write
\begin{align}
\nonumber
\kappa_k (X_T) &= \frac{\partial^k   }{\partial t^k}  \sum_{\Gamma \in \cC}  \tilde w(\Gamma) \Bigg |_{t=0} \\
\label{eqClusterKappa}
&= \sum_{\Gamma \in \cC} w(\Gamma)  Y_T(\Gamma)^k   \, ,
\end{align}
where $Y_T(\Gamma) = \sum_{S \in \Gamma } \mathbf 1_{\tau(S) =T}$,
the number of polymers of type $T$ in the cluster $\Gamma$.

The following lemma gives bounds on cluster weights using the \kotecky--Preiss condition.  Theorem~\ref{thmDistribution} will then follow in a series of corollaries.  
\begin{lemma}
\label{lemClustergammaBound}
Consider a fixed defect type $T$, and let $k \ge 1$ be a fixed integer.  Then
\begin{align}
\label{eqOneTbound}
\sum_{\Gamma \in \cC} w(\Gamma) Y_{T}(\Gamma)^k &=  (1+o(1)) n_T w_T
\end{align}
as $d \to \infty$. 

Moreover if $\{ T_1, \dots, T_{\ell} \}$ is a fixed set of distinct defect types, and $k_1, \dots k_\ell$ are fixed positive integers, then
\begin{align}
\label{lemClustergammaBound2}
\sum_{\Gamma \in \cC} | w(\Gamma) | \prod_{i=1}^\ell Y_{T_i}(\Gamma)^{k_i}  &= O \left(   d^{7 \ell}  2^d   \prod_{i=1}^\ell w_{T_i}      \right )\,.
\end{align}
\end{lemma}
\begin{proof}
In the sum in~\eqref{eqOneTbound}, if we consider only clusters made up of a single polymer of type $T$ then we get a contribution of exactly $n_T w_T$, and so it remains to show that the contribution of all other terms is $o( n_T w_T)$.  
Let $t$ denote the number of vertices in a graph of type $T$.  We first consider the contribution to the sum \eqref{eqOneTbound} from clusters $\Gamma$ with $Y_T(\Gamma) =1$ and $\|\Gamma \| > t$. By \eqref{eqClusterBound}, we may bound this contribution by
\begin{align}
\sum_{\substack{\Gamma \in \cC \\ \|\Gamma \| \ge t+1}} |w(\Gamma)| &\le d^{-3/2} 2^{d-1} e^{-\gamma(d, t+1)}= d^{11/2} 2^{d-1}(1+\lam)^{-d(t+1)+3(t+1)^2}= o(n_T w_T) \,,
\end{align}
since from~\eqref{eqWtBound} and~\eqref{eqdefect1}
\begin{align*}
n_T w_T\ge \frac{2^{d-1} \lam^t}{t (1+\lam)^{dt}} \, .
\end{align*}

Consider now the contribution to the sum \eqref{eqOneTbound} from clusters $\Gamma$ with $Y_T(\Gamma) =y>1$. For such a cluster, recalling that $g(S)=\gamma(d,|S|)$ for a polymer $S$ and $g(\Gamma)=\sum_{S\in \Gamma}g(S)$, we have
\begin{align}
g(\Gamma)\ge y [\log(1+\lam)(dt-3t^2)-7\log d]\, .
\end{align}
Thus, using \eqref{eqgtail} we may bound this contribution by
\begin{align}
2^{d-1} d^{-3/2} e^{-y [\log(1+\lam)(dt-3t^2)-7\log d]} y^k =  y^k d^{7y-3/2} 2^{d-1} (1+\lam)^{-dyt+3yt^2}\le 2^{d-1}(1+\lam)^{-3dyt/4} \,,
\end{align}
where the above inequality holds for $d$ large enough (independent of $y$).
The result follows since
\begin{align*}
\sum_{y=2}^\infty 2^{d-1}(1+\lam)^{-3y dt/4}\le 2^{d} (1+\lam)^{-3dt/2}=o(n_T w_T).
\end{align*}

Next we turn to~\eqref{lemClustergammaBound2}. Consider a cluster $\Gamma$ with $Y_{T_1}(\Gamma) =y_1, \dots , Y_{T_\ell}(\Gamma)= y_\ell$, where $y_1, \dots, y_\ell \ge 1$.  Then we have 
\begin{align*}
g(\Gamma) &\ge \sum_{j=1}^\ell y_j [\log(1+\lam)(dt_j-3t_j^2)-7\log d]
\end{align*}
where $t_j$ is the size of a polymer of type $T_j$.  By \eqref{eqgtail} and \eqref{eqWtBound}, the contribution of such clusters to the sum in~\eqref{lemClustergammaBound2} is therefore at most 
\begin{align*}
\frac{2^{d-1}}{ d^{3/2}} \prod_{j=1}^\ell y_j^{k_j}  d^{7 y_j} (1+\lam)^{-d t_j y_j + 3t_j^2 y_j  } &=  O\left (2^d d^{7 \ell} \prod_{j=1}^\ell w_{T_j}     \right)   \prod_{j=1}^\ell y_j^{k_j}  d^{7 (y_j-1)} (1+\lam)^{-d t_j (y_j-1) + 3t_j^2 y_j  } \,.
\end{align*}
Finally, let $K = \max \{k_1, \dots k_\ell, t_1, \dots t_\ell\}$, so that  summing over all positive integer vectors $\vec y=(y_1, \dots, y_\ell)$, we have
\begin{align*}
\sum_{\vec y}  \prod_{j=1}^\ell y_j^{k_j}  d^{7 (y_j-1)} (1+\lam)^{-d t_j (y_j-1) + 3t_j^2 y_j  }   &\le  \sum_{\vec y}  \prod_{j=1}^\ell y_j^{K}  d^{7 (y_j-1)} (1+\lam)^{-d  (y_j-1) + 3K^2 y_j  }  \\
&\le \sum_{s=0}^\infty \sum_{\substack{ \vec y: \\ \sum (y_j-1) = s}} d^{7s} (1+\lam)^{-ds} \prod_j y_j^K (1+\lam)^{3K^2 y_j} \\
&= (1+\lam)^{3K^2 \ell} \sum_{s=0}^\infty \sum_{\substack{ \vec y: \\ \sum (y_j-1) = s}} d^{7s} (1+\lam)^{-(d-3K^2)s} \prod_j y_j^K\\
&\le (1+\lam)^{3K^2 \ell} \sum_{s=0}^\infty \sum_{\substack{ \vec y: \\ \sum (y_j-1) = s}} (s+\ell)^K d^{7s} (1+\lam)^{-(d-3K^2)s}\\
&\le (1+\lam)^{3K^2 \ell} \sum_{s=0}^\infty s^\ell (s+\ell)^K d^{7s} (1+\lam)^{-(d-3K^2)s}\\
 &= O(1) \,.
\end{align*}
Putting these estimates together yields~\eqref{lemClustergammaBound2}. 
\end{proof}

An immediate corollary of Lemma~\ref{lemClustergammaBound} gives the asymptotics of $m_{T}$, $\sigma^2_{T}$ for a given type $T$. 
\begin{cor}
\label{lemMeanVar}
Let $T$ be a defect type.  Then
\begin{align*}
m_{T} &= (1+o(1)) n_T w_T
\intertext{and}
\sigma^2_{T} &= (1+o(1)) n_T w_T\,.
\end{align*}
\end{cor}
\begin{proof}
These formulae follow from~\eqref{eqClusterKappa} and~\eqref{eqOneTbound} by taking $k=1$ and $k=2$ respectively. 
\end{proof}

We can also use Lemma~\ref{lemClustergammaBound} to prove Poisson convergence.  
\begin{cor}
\label{lemPoisson}
Suppose for a given type $T$ and fugacity $\lam$ we have $m_{T} \to \rho>0$ as $d \to \infty$.  Then $X_{T} \Rightarrow \text{Pois}(\rho)$.  
\end{cor}
\begin{proof}
Using Fact~\ref{factPoisNormal}, it is enough to show that $\kappa_k(X_{T}) \to \rho$ for all $k\ge 1$.  By~\eqref{eqClusterKappa} and our assumption we have 
\begin{align*}
m_T &= \sum_{\Gamma \in \cC} w(\Gamma) Y_{T}(\Gamma) =  \rho +o(1),
\end{align*}
and therefore using~\eqref{eqClusterKappa} again,
\begin{align*}
\sum_{\Gamma \in \cC} w(\Gamma) Y_{T}(\Gamma)^k &= \rho +o(1) 
\end{align*}
for all $k \ge 1$.  
\end{proof}

In a similar fashion, we obtain asymptotic normality if $m_T \to \infty$. 
\begin{cor}
\label{lemNormal}
Fix a type $T$. If $\lam$ is such that  $m_{T} \to \infty$ as $d \to \infty$,   then $\tilde X_{T} = (X_{T}-m_{T})/\sigma_{T} \Rightarrow \text{N}(0,1)$.  
\end{cor}
\begin{proof}
By Fact~\ref{factPoisNormal},  it suffices to show that $\kappa_1(\tilde X_{T}) \to 0$, $\kappa_2(\tilde X_{T}) \to 1$, and $\kappa_k(\tilde X_{T}) \to 0$ for all $k \ge 3$.  By the definition of $\tilde X_{T}$, we have $\kappa_1(\tilde X_{T}) = 0$ and $\kappa_2(\tilde X_{T}) = 1$.  By translation invariance and scaling of higher cumulants, for $k \ge 3$ we have 
\begin{align*}
\kappa_k (\tilde X_{T}) &= \frac{1}{\sigma_{T}^k}  \kappa_k(X_{T}) \\
&= \frac{1}{\sigma_{T}^k}  \sum_{\Gamma \in \cC} w(\Gamma)  Y_{T}(\Gamma)^k 
\end{align*}
by~\eqref{eqClusterKappa}.  By Lemmas~\ref{lemClustergammaBound} and~\ref{lemMeanVar} we have $\sum_{\Gamma \in \cC} w(\Gamma)  Y_{T}(\Gamma)^k = (1+o(1)) \sigma^2_{T}$, and so for $k \ge 3$,
\begin{align*}
\kappa_k (\tilde X_{T}) &= O\left(  \sigma^{2-k}_{T} \right) \to 0
\end{align*}
as $d \to \infty$ since our assumption on $m_T$ implies $\sigma_{T} \to \infty$.
\end{proof}

To study the joint distribution of the counts of different defect types, it is convenient to work with the \textit{joint cumulants} of a collection of random variables.  Given a set of random variables $(X_1, \dots, X_\ell)$ and non-negative integers $k_1, \dots, k_\ell$, we define  the joint cumulant 
\begin{align*}
\kappa \left(X_1^{(k_1)}, \dots , X_\ell^{(k_{\ell})} \right) &= \frac{ \partial^{\sum_i k_i}     }{ \prod_{i} \partial t_i^{k_i}  } \log \E e^{\sum_{i=1}^\ell t_i X_i}  \Big|_{t_1, \dots, t_\ell =0 } \,.
\end{align*}
In particular, with this notation
\begin{align*}
\kappa_k(X) = \kappa(X^{(k)} )  \, .
\end{align*}
We will use the fact that the joint cumulants of independent random variables vanish; that is, if  $\ell \ge 2$, $X_1, \dots, X_{\ell}$ are independent random variables, and $k_1, \dots, k_{\ell}$ are positive integers, then 
\begin{equation}\label{eqJointVanish}
\kappa \left(X_1^{(k_1)}, \dots , X_\ell^{(k_{\ell})} \right) = 0 \, .
\end{equation}

Generalizing formula~\eqref{eqClusterKappa} to collections of random variables, we can express the joint cumulants of defect type counts via a modified cluster expansion.  Let $ \{ T_1, \dots, T_\ell \}$ be a set of distinct defect types and let $k_1, \dots, k_\ell$ be non-negative integers.  Then
\begin{align*}
\kappa \left(X_{T_1}^{(k_1)}, \dots , X_{T_\ell}^{(k_{\ell})} \right) &=  \sum_{\Gamma \in \cC} w(\Gamma) \prod_{i=1}^\ell Y_{T_i} (\Gamma)^{k_i} \,.
\end{align*}

\begin{cor}
\label{lemJointNormal}
Consider two fixed sets $\cT_1$ and $\cT_2$ of distinct defect types so that for each $T \in \cT_1$, $m_T  \to \rho_T$ for some $\rho_T>0$, and for each $T \in \cT_2$, $m_T \to \infty$ as $d \to \infty$.  Then the collection of random variables $\{ X_T \}_{T \in \cT_1} \cup \{ \tilde X_T \}_{T \in \cT_2}$ converges in distribution to a collection of independent Poisson and standard normal random variables.
\end{cor}
\begin{proof}
We will use the fact that the distribution of a collection of Poisson and  normal random variables is determined by its joint moments, or equivalently, by its joint cumulants. Here, working with cumulants instead of moments will simplify calculations considerably.  From Corollaries~\ref{lemPoisson} and~\ref{lemNormal} we know that the cumulants of each of the individual random variables in the collection converge to the corresponding  cumulants of the corresponding Poisson or normal random variable.    Therefore it is enough to show convergence of the joint cumulants involving at least two of the random variables, and from~\eqref{eqJointVanish}, we must show that these converge to $0$. In particular, for $T_1, \dots, T_j \in \cT_1$, and $T_{j+1}, \dots , T_\ell \in \cT_2$, we will show
\begin{equation}
\label{eqKappaJoint1}
\kappa( X_{T_1}^{(k_1)}, \dots, X_{T_j}^{(k_j)}, \tilde X_{T_{j+1}}^{(k_{j+1})}, \tilde X_{T_\ell}^{(k_\ell)})  \to 0  
\end{equation}
as $d \to \infty$ as long as least two of the $k_i$'s are positive.  Since $\sigma^2_T \to \rho_T >0$ for $T \in \cT_1$, it will suffice to show~\eqref{eqKappaJoint1} when we center and normalise all of the random variables, that is, for $T_1, \dots , T_\ell \in \cT_1 \cup \cT_2$,
\begin{equation}
\label{eqKappaJoint}
\kappa( \tilde X_{T_1}^{(k_1)}, \dots, \tilde X_{T_\ell}^{(k_\ell)})  \to 0  
\end{equation}
as long as at least two of the $k_i$'s are positive.
   WLOG we can assume that $\ell \ge 2$,  $k_i \ge 1$ for all $i$, and that $w_{T_1} \ge w_{T_2} \ge \cdots \ge w_{T_\ell}$.  By scaling and translation invariance, we have 
\begin{align*}
\kappa(\tilde X_{T_1}^{(k_1)}, \dots, \tilde X_{T_\ell}^{(k_\ell)}) &= \prod_{i=1}^\ell \frac{1}{\sigma_{T_i}^{k_i}} \kappa(X_{T_1}^{(k_1)}, \dots, X_{T_{\ell}}^{(k_\ell)} ) \\
&=  \prod_{i=1}^\ell \frac{1}{\sigma_{T_i}^{k_i}} \sum_{\Gamma}  w(\Gamma) \prod_{i=1}^\ell Y_{T_i} (\Gamma)^{k_i} \,.
\end{align*}
 Then using~ \eqref{lemClustergammaBound2} from Lemma~\ref{lemClustergammaBound} we have 
 \begin{align*}
\left| \kappa(\tilde X_{T_1}^{(k_1)}, \dots, \tilde X_{T_\ell}^{(k_\ell)}) \right| &=O \left( \prod_{i=1}^\ell \frac{1}{\sigma_{T_i}^{k_i}}    \cdot d^{7 \ell} \cdot 2^d \prod_{i=1}^\ell w_{T_i}   \right)  \, .
\end{align*}
First suppose that $k_1 \ge 2$.  Then since $\sigma_{T_i} =\Omega(1)$ for all $i$ and $2^d w_{T_1} = O(\sigma^2_{T_1})$,  we have 
\begin{align*}
\left| \kappa(\tilde X_{T_1}^{(k_1)}, \dots, \tilde X_{T_\ell}^{(k_\ell)}) \right| &=  O  \left( d^{7 \ell} \sigma_{T_1}^{2 - k_1}  \prod_{i=2}^{\ell} w_{T_i} \sigma_{T_i}^{-k_i}      \right) \\
&=  O  \left( d^{7 \ell}  w_{T_2}     \right)  \\
&= o(1)
\end{align*}
since $w_T$ tends to $0$ faster than any fixed polynomial in $d$ for any type $T$.  
On the other hand if we have $k_1=1$, then
\begin{align*}
\left| \kappa(\tilde X_{T_1}^{(k_1)}, \dots, \tilde X_{T_\ell}^{(k_\ell)}) \right| &=  O  \left(d^{7 \ell} w_{T_1}^{1/2} w_{T_2}^{1/2}  \sigma_{T_2}^{1-k_2}  \prod_{i=3}^{\ell} w_{T_i}  \sigma_{T_i}^{-k_i} \right )  \\
&=  O  \left(d^{7 \ell}  w_{T_1}^{1/2}     \right)  \\
&= o(1) \,.
\end{align*}
\end{proof}

Theorem~\ref{thmDistribution} follows from Corollaries~\ref{lemPoisson}, \ref{lemNormal}, and~\ref{lemJointNormal}.  We now prove Theorem~\ref{corEmergence}.

\begin{proof}[Proof of Theorem~\ref{corEmergence}]
We can assume in what follows that  $\lam \le 2$, since if $\lam >2$ whp there are no occupied vertices on the minority side (Theorem 1.2 of \cite{galvin2011threshold}). 

First we show that if $\lam =2^{1/t} -1 + \frac{  2^{1+1/t}(t-1)\log d}{td} + \frac{\omega(1)}{d} $, then whp there are no $2$-linked components of size $t$ on the defect side. 

Let $T$ the type of a polymer of size $t$.  We have the bounds
\begin{align*}
 \frac{\lam^t}{(1+\lam)^{td}} \le w_T \le \frac{\lam^t}{(1+\lam)^{td - 2t^2}} \,,
\end{align*}
where the upper bound uses Lemma~\ref{lemIso}, 
and so $w_T = \Theta(\lam^t (1+\lam)^{-dt}) = \Theta((1+\lam)^{-dt})$ for this range of $\lam$.   By Lemma~\ref{lemlink}, $n_T = O(2^d d^{2 (t-1)})$, and so by Corollary~\ref{lemMeanVar}, 
\begin{align*}
m_T = O\left ( \frac{ 2^d d^{2(t-1)} }{(1+\lam)^{dt}} \right) \,.
\end{align*}
Now plugging in $\lam = 2^{1/t} -1 + \frac{  2^{1+1/t}(t-1)\log d}{td} + \frac{s}{d}$ for some $s$, we have 
\begin{align*}
m_T &= O \left( \frac{ 2^d d^{2(t-1)} }{\left(2^{1/t} + \frac{  2^{1+1/t}(t-1)\log d}{td} + \frac{s}{d}\right)^{dt}}   \right) \\
&= O \left( \frac{ d^{2(t-1)} }{\left(1 + \frac{  2(t-1)\log d}{td} + \frac{s2^{-1/t}}{d}\right)^{dt}}   \right) \\
&= O \left(  e^{-st 2^{-1/t}} \right )  \,,
\end{align*}
and so as $s \to \infty$, $m_T \to 0$.  This is true for any type $T$ of size $t$, and since there are a constant number of such types, Markov's inequality shows that whp there are no polymers of size $t$ in $\mathbf \Gamma$ if $\lam =2^{1/t} -1 + \frac{  2^{1+1/t}(t-1)\log d}{td} + \frac{\omega(1)}{d} $. 

Now suppose  $\lam =2^{1/t} -1 + \frac{  2^{1+1/t}(t-1)\log d}{td} - \frac{\omega(1)}{d} $.  Consider a type $T$ where $T$ is isomorphic to a tree on $t$ vertices.
 In this case we have $n_T = \Theta(2^d d^{2(t-1)})$ by Lemma~\ref{lemDefectTypes}, and so for $\lam =  2^{1/t} -1 + \frac{  2^{1+1/t}(t-1)\log d}{td} + \frac{s}{d}$ the previous calculation gives
\begin{align*}
m_T = \Omega\left(  e^{-st 2^{-1/t}}   \right) \,.
\end{align*}
In particular if $s \to - \infty$ as $d\to\infty$ then $m_T \to \infty$.  By Corollary~\ref{lemMeanVar}, $\sigma^2_T \sim m_T$, and so by the second-moment method (Paley-Zygmund inequality), $X_T \ge 1$ whp.  

To prove the second part of Theorem~\ref{corEmergence}, suppose that  $\lam = 2^{1/t} -1 + \frac{  2^{1+1/t}(t-1)\log d}{td} + \frac{s(d)}{d}$ where $s(d)$ converges to a constant $s$ as $d \to \infty$.  Then for any type $T$ of size $t$ that is not a tree, by Lemma~\ref{lemDefectTypes} we have $m_T = o(1)$ as $d \to \infty$, and since there is a constant number of such types, we know that whp there are no non-tree $2$-linked components of size $t$ on the minority side.  Let $T_1, \dots , T_\ell$ be the defect types of size $t$ that are trees. The proof of Lemma~\ref{lemDefectTypes} shows that in fact every tree on $t$ vertices is a defect type.  Note that for each $i$ we have that $w_{T_i}=\lam^t (1+\lam)^{-dt+2(t-1)}$. Then by Lemma~\ref{lemDefectTypes} and Corollary~\ref{lemMeanVar} we have that 
$m_{T_i} = (c_{T_i}+o(1))( \lam^t (1+\lam)^{-dt+2(t-1)} 2^d d^{2 (t-1)})$
 for each $i$, and so by a similar calculation as above we have that $m_{T_i} \to \rho_i$ as $d \to \infty$ where 
 \[
 \rho_i= \frac{1}{2^t|\text{Aut}(T_i)|}e^{-st 2^{-1/t}} (2^{1/t}-1)^t 2^{2(1-1/t)}\, .
 \]
 By Corollary~\ref{lemJointNormal}, the collection of random variables $X_{T_1}, \dots , X_{T_\ell}$ converges to a collection of independent Poisson random variables with mean $\rho_1, \dots, \rho_{\ell}$, and therefore their sum is distributed as Poisson with mean $\sum_{i=1}^\ell \rho_i$, completing the proof of Theorem~\ref{corEmergence}.  Calculating this mean explicitly amounts to calculating $|\text{Aut}(T)|$ for every tree $T$ on $t$ vertices, a task whose running time depends only on $t$ (a constant).
 \end{proof}

The proof of Corollary~\ref{corZasymptotics} involves a similar calculation.
\begin{proof}[Proof of Corollary~\ref{corZasymptotics}]

We may again assume that $\lam \le 2$ since for larger $\lam$, $Z(\lam) = (2+o(1)) (1+\lam)^{2^{d-1}}$ by Theorem~\ref{thmGalvin}.  Now fix $ t \ge 1$ and take $\lam = 2^{1/t} -1 + \frac{  2^{1+1/t}(t-1)\log d}{td} + \frac{\omega(1)}{d}$. We then can apply Theorem~\ref{thmHypercubeasymptotics} with $k=t $ to obtain
\begin{align*}
Z(\lam) &=  2 (1+\lam)^{2^{d-1}}  \cdot \exp \left ( \sum_{j=1}^t L_j    + \eps_t \right)  \,.
\end{align*}
But by the same calculation as above in the proof of Theorem~\ref{corEmergence} we have 
\begin{align*}
|L_t| &= O\left(\frac{2^d  d^{2(t-1)}}{  (1+\lam)^{dt} } \right)\\
&= o(1) \, ,
\end{align*}
and
\begin{align*}
|\eps_t| &= O\left( \frac{2^d  d^{2t} }{(1+\lam)^{d(t+1)   }}  \right ) = o(1) \, ,
\end{align*}
and so 
\begin{equation}
Z(\lam) =  (2+o(1)) (1+\lam)^{2^{d-1}}   \exp \left ( \sum_{j=1}^{t-1} L_j    \right)  \,.
\end{equation}

The example formula~\eqref{eqZexample} follows from the computation of $L_1, L_2$ given below in Section~\ref{secClusterWeights}. 
\end{proof}

\section{Computation of the cluster weights}
\label{secClusterWeights}

Here we compute $L_1, L_2, L_3$ explicitly to use in Theorem~\ref{ThmRefined} and Corollary~\ref{corZasymptotics}.

\begin{prop}
\label{propLs}
We have
\begin{align*}
L_1&=    \frac{2^d \lam}{(1+\lam)^d} \cdot \frac{1}{2} \\
L_2 &=    \frac{2^d \lam^2}{(1+\lam)^{2d}}  \cdot  \frac{  (2\lam +\lam^2) d(d-1) -2    }{  8   }  
\end{align*}
\begin{align*}
 L_3 =   \frac{2^d \lam^3}{48(1+\lam)^{3d}} \big[8 &+ 2 ({8 \lam} - 2 \lam^2 + 4 \lam^3 + 11 \lam^4 + 4 \lam^5) d + 
 3 ({-4 \lam }+ 12 \lam^2 + 4 \lam^3 - 9 \lam^4 - 4 \lam^5) d^2 \\&+
 2 ({-2 \lam }- 22 \lam^2 - 16 \lam^3 + \lam^4 + 2 \lam^5) d^3 + 
 3 (4 \lam^2 + 4 \lam^3 + \lam^4) d^4  \big]
\end{align*}
At $\lam =1$, this is
\begin{align*}
L_1 &=     \frac{1}{2} \\
L_2 &=    2^{-d} \cdot \frac{3d^2 - 3d -2  }{ 8  }  \\
L_3 &=     2^{-2d} \cdot  \frac{27 d^4 - 74 d^3 -3 d^2 + 50d +8}{48}    \,.
\end{align*}
\end{prop}

\paragraph{Polymers}

There is a single type of polymer of size $1$.  There are $2^{d-1}$ of these, and each has weight  $\lam (1+\lam)^{-d}$. 

There is a single type of polymer of size $2$. There are $2^{d-3} d (d-1)$ of these and each has weight $\lam^2 (1+\lam)^{-2d+2}$. 

There are two types of polymers of size $3$: those that form a clique in the distance $2$ graph and those that form a path on $3$ vertices.  There are $2^{d-2} d(d-1)(d-2)/3$ of the first type and each has weight $\lam^3 (1+\lam)^{-3d +5}$; there are $2^{d-4} d(d-1)(d-2)(d-3)$ of the second type and each has weight $\lam^3 (1+\lam)^{-3d+4}$.

\paragraph{Clusters}

There is a single cluster type of size $1$, each consisting of single polymer of size $1$, with Ursell function $1$. Thus
\begin{align}
L_1=\frac{2^d \lam}{(1+\lam)^d} \cdot \frac{1}{2} 
\end{align}

There are two types of clusters of size $2$: an ordered pair of incompatible polymers of size $1$, of which there are $2^{d-1} + 2^{d-2} d (d-1) $, with Ursell function $-1/2$ and weight $ \lam^2 (1+\lam)^{-2d}$, and one polymer of size $2$ with Ursell function $1$ and count and weight given above.

All together this gives:
\begin{align*}
L_2 &=   - \frac{1}{2}\left(2^{d-1} + 2^{d-2} d (d-1) \right)  \lam^2 (1+\lam)^{-2d}   +  2^{d-3} d (d-1) \lam^2 (1+\lam)^{-2d+2} \\
&=  \frac{2^d \lam^2}{(1+\lam)^{2d}}  \cdot  \frac{  (2\lam +\lam^2) d(d-1) -2    }{  8   } \, .
\end{align*}
At $\lam=1$ this is 
\begin{align*}
2^{-d} \cdot \frac{3d^2 - 3d -2  }{ 8  }\, .
\end{align*}

There are five types of clusters of size $3$: 
\begin{enumerate}
\item One polymer of size $3$, first type: $2^{d-2} d(d-1)(d-2)/3$ of weight $\lam^3 (1+\lam)^{-3d +5}$, Ursell function $1$.  
\item One polymer of size $3$, second type: $2^{d-4} d(d-1)(d-2)(d-3)$ of weight $\lam^3 (1+\lam)^{-3d+4}$, Ursell function $1$. 
\item Three polymers of size $1$, incompatibility graph is a triangle: $2^{d-1} + 3\cdot2^{d-2}d(d-1)+2^{d-1} d(d-1)(d-2)$  of weight $\lam^3 (1+\lam)^{-3d}$, Ursell function $1/3$. 
\item Three polymers of size $1$, incompatibility graph is a path on $3$ vertices: $3\cdot2^{d-3} d(d-1)(d-2)(d-3)$  of  weight $\lam^3 (1+\lam)^{-3d}$, Ursell function $1/6$.   
\item One polymer of size $2$, one of size $1$: $2^{d-2} d (d-1)[d(d-1)-2(d-2)]$  of weight $\lam^3 (1+\lam)^{-3d+2}$, Ursell function $-1/2$. 
\end{enumerate}

All together this gives:
\begin{align*}
 L_3 =   \frac{2^d \lam^3}{48(1+\lam)^{3d}} \big[8 &+ 2 ({8 \lam} - 2 \lam^2 + 4 \lam^3 + 11 \lam^4 + 4 \lam^5) d + 
 3 ({-4 \lam }+ 12 \lam^2 + 4 \lam^3 - 9 \lam^4 - 4 \lam^5) d^2 \\&+
 2 ({-2 \lam }- 22 \lam^2 - 16 \lam^3 + \lam^4 + 2 \lam^5) d^3 + 
 3 (4 \lam^2 + 4 \lam^3 + \lam^4) d^4  \big]
\end{align*}

At $\lam=1$ this is
\begin{align*}
2^{-2d} \cdot  \frac{27 d^4 - 74 d^3 -3 d^2 + 50d +8}{48}
\end{align*}

\begin{proof}[Proof of Theorem~\ref{ThmRefined}]
Theorem~\ref{thmHypercubeasymptotics} tells us that 
\begin{align*}
i(Q_d) &=  2 \cdot 2^{2^{d-1}} \cdot \exp \left(L_1 + L_2 + L_3 + O(L_4)     \right)  \\
&= 2\sqrt{e} \cdot 2^{2^{d-1}} \cdot \exp \left(L_2 + L_3 + O(L_4)     \right)
\end{align*}
since $L_1 =1/2$.  If we write $L_k = a_{k-1} 2^{-(k-1)d}$, then we have 
\begin{align*}
i(Q_d) &= 2\sqrt{e} \cdot 2^{2^{d-1}} \cdot \exp \left ( a_1 2^{-d} +a_2 2^{-2d} +O(a_3 2^{-3d}) \right ) \,.
\end{align*}
Since the Taylor series for $\exp(a_1 x + a_2 x^2 +a_3 x^3)$ around $x=0$ is 
$$1 + a_1 x + \left ( \frac{a_1^2}{2} +a_2  \right)x^2 + O\left( \left(a_1^3 + a_1 a_2 + a_3  \right) x^3\right) \, ,$$
we have
{\small
\begin{align*}
i(Q_d) &= 2\sqrt{e} \cdot 2^{2^{d-1}}  \left(   1+ L_2   +\frac{L_2^2}{2} + L_3  + O(L_2^3 + L_2 L_3 + L_4 )   \right)   \\
&=  2\sqrt{e} \cdot 2^{2^{d-1}}  \left(   1+ \frac{3d^2 - 3d -2  }{8 \cdot 2^{d}}  +  {  \frac{243 d^4 - 646 d^3 - 33 d^2 + 436d +76}{384 \cdot 2^{2d}}  }+ O\left(d^ 6 \cdot 2^{-3d} \right)   \right )
\end{align*}
}
which gives Theorem~\ref{ThmRefined}. 
\end{proof}

\section*{Acknowledgements}
WP supported in part by NSF Career award 1847451.  Part of this work was done while WP visiting the Simons Institute for the Theory of Computing.  We thank Lina Li for some very helpful comments on this paper.

\end{document}